\documentclass[12pt,a4paper]{amsart}
\usepackage{amsfonts}
\usepackage[active]{srcltx}
\usepackage[notref,notcite]{showkeys}
\usepackage{enumerate}

\usepackage{amsmath,amssymb,xspace,amsthm}
\newtheorem{theorem}{Theorem}
\newtheorem{example}{Example}
\newtheorem{remark}[theorem]{Remark}
\newtheorem{lemma}[theorem]{Lemma}
\newtheorem{proposition}[theorem]{Proposition}
\newtheorem{corollary}[theorem]{Corollary}
\newtheorem{case}{Case}
\newtheorem{subcase}{Subcase}[case]
\numberwithin{equation}{section}

\def\Vir{\mathfrak{V}}

\def\Der{\operatorname{Der}}
\def\span{\operatorname{span}}

\newcommand{\C}{\ensuremath{\mathbb C}\xspace}

\renewcommand{\a}{\ensuremath{\alpha}}

\newcommand{\ord}{{\rm ord}}

\newcommand{\Z}{\ensuremath{\mathbb{Z}}\xspace}
\newcommand{\N}{\ensuremath{\mathbb{N}}\xspace}

\newcommand{\Ind}{\ensuremath{\operatorname{Ind}}\xspace}
\newcommand{\Soc}{\ensuremath{\operatorname{Soc}}\xspace}

\newcommand{\ann}{\operatorname{ann}\xspace}

\renewcommand{\phi}{\varphi}

\renewcommand{\geq}{\geqslant}

\begin{document}
\title[A family of simple weight Virasoro modules]{A family of simple weight modules over the Virasoro algebra}
\author{Rencai L{\"u} and Kaiming Zhao}
\date{Mar.3, 2013}
\maketitle

\begin{abstract} Using simple modules over  the derivation Lie algebra $\C[t]\frac{d}{d t}$
of the associative polynomial algebra $\C[t]$, we construct new
weight  Virasoro modules with all weight spaces infinite
dimensional. We determine necessary and sufficient conditions for
these new weight Virasoro modules to be simple, and determine
necessary and sufficient conditions for two such weight Virasoro
modules to be isomorphic. If such a weight Virasoro module is not
simple, we obtain all its submodules. In particular, we completely
determine the simplicity and the isomorphism classes of the weight
modules defined in \cite{CM} which are a small portion of the
modules constructed in this paper.
\end{abstract}


\vskip 5pt \noindent {\em 2000  Math. Subj. Class.:} 17B10, 17B20,
17B65, 17B66, 17B68

\vskip 10pt

\section{Introduction}
We denote by $\mathbb{Z}$, $\mathbb{Z}_+$, $\N$, and $\mathbb{C}$
the sets of  all integers, nonnegative integers, positive integers,
 and complex numbers, respectively. For a Lie algebra
$L$ we denote by $U(L)$ the universal enveloping algebra of $L$.

 The Virasoro algebra $\mathfrak{V}$ is the universal
central extension of the derivation algebra of the Laurent
polynomial algebra $\C[t,t^{-1}]$. More precisely, $\mathfrak{V}$ is
a Lie algebra over $\C$ with the basis
$$\{t^{n+1}\frac{d}{d t},z| n \in \Z\}$$
and subject to the Lie bracket
\begin{equation}[t^{n+1}\frac{d}{d t},t^{m+1}\frac{d}{d t}]=(m-n)t^{m+n+1}\frac{d}{d t}
+\delta_{n,-m}\frac{n^3-n}{12}z,\end{equation}
\begin{equation}[\mathfrak{V},z]=0. \end{equation}
Denote $d_n=t^{n+1}\frac{d}{d t }$. We will used both notations
according to contexts.

The Virasoro algebra is one of the most important Lie algebras both
in mathematics and in mathematical physics, see for example
\cite{KR,IK} and references therein. Its theory has been widely used
in many physics areas and other mathematical branches, for example,
quantum physics \cite{GO}, conformal field theory \cite{FMS},
Higher-dimensional WZW models \cite{IKUX, IKU}, Kac-Moody algebras
\cite{K, MoP}, vertex algebras \cite{LL}, and so on.

The representation theory on the Virasoro algebra has been
attracting a lot of attentions from mathematicians and physicists.
There are two classical families of simple Harish-Chandra
$\mathfrak{V}$-modules: highest weight modules (completely described
in \cite{FF}) and the so-called intermediate series modules. In
\cite{Mt} it is shown that these two families exhaust all simple
weight Harish-Chandra modules. In \cite{MZ1} it is even shown that
the above modules exhaust all simple weight modules admitting a
nonzero finite dimensional weight space.

Very naturally, the next important task is to study simple weight
modules with infinite dimensional weight spaces. The first such
examples were constructed by taking  the tensor product of some
highest weight modules and some intermediate series modules in
\cite{Zh} in 1997, and the necessary and sufficient conditions for
such tensor product to be simple were recently obtained in
\cite{CGZ}. Conley and Martin gave another class of such examples
with four parameters in \cite{CM} in 2001 where some sufficient
conditions were discussed for the modules to be simple. Then very
recently, a big class of  weight simple Virasoro modules were found
in \cite{LLZ}. We remark that the tensor products of intermediate
series modules over the Virasoro algebra never gives irreducible
modules \cite{Zk}. In this paper, we construct a family of weight
simple Virasoro modules which include all the modules defined in
\cite{CM, LLZ} as a small portion.

At the same time for the last decade, various other families of nonweight simple
$\mathfrak{V}$-modules were studied in \cite{OW,LGZ, LZ,
FJK,Ya,GLZ,OW2, MW, TZ}. These include   various versions of Whittaker
modules  constructed using different tricks. In particular, all the
above Whittaker modules and even more were described in a uniform
way in \cite{MZ3}.

To introduce the contents of the present paper we  need to   define
the following subalgebras of $\mathfrak{V}$ where $r\in\Z_+$:
\begin{equation}\mathfrak{W}=\Der(\C[t])=\span \{d_i|i\ge
-1\},\end{equation} \begin{equation}\mathfrak{b}=\span \{d_i |i\ge
0\},\end{equation}
\begin{equation}\mathfrak{V}^{(r)}=\span \{d_i|i\ge r\},\end{equation}
\begin{equation}\mathfrak{a}_r=\mathfrak{b}/\mathfrak{V}^{(r+1)}.\end{equation}
 The Lie algebra $\mathfrak{W}$ is
usually called the Witt algebra of rank one. We denote by
$\mathcal{O}_{\mathfrak{W}}$ the category of all $\mathfrak{W}$-modules
$W$ satisfying

 {\bf Condition A:} For any
$w\in W$, there exists a nonnegative integer $n$ depending on $W$
such that $d_i w=0$ for all $i\ge n$.

 Similarly we may define the  categories $\mathcal{O}_{\mathfrak{V}}$, $\mathcal{O}_{\mathfrak{b}}$.
It is clear that  $\mathcal{O}_{\mathfrak{V}}$ consists of highest
weight modules and the ones define in \cite{MZ3}.

 The paper is organized as follows. In Sect.2, we determine all simple modules in
 $\mathcal{O}_{\mathfrak{b}}$ and all simple modules in
 $\mathcal{O}_{\mathfrak{W}}$. Actually, a simple modules in
 $\mathcal{O}_{\mathfrak{b}}$ is a simple module over $\mathfrak{a}_r$ for some $r\in\N$, and
 all nontrivial simple modules in
 $\mathcal{O}_{\mathfrak{W}}$ are induced modules from a simple module over $\mathfrak{a}_r$.
 In Sect.3, for any $W\in \mathcal{O}_{\mathfrak{W}}$, $a, b\in\C$ and  $\lambda\in \C^*$,
we define our weight Virasoro modules
$\mathcal{L}(W,\lambda,a,b)=W\otimes \C[t,t^{-1}]$, and   we prove
that all the Virasoro  modules $E_h(b,\gamma,p)$ defined in
\cite{CM} are only very special cases of the modules
$\mathcal{L}(W,\lambda,a,b)$. We also establish a powerful technique
for later use. In Sect.4 we determine necessary and sufficient
conditions for $\mathcal{L}(W,\lambda,a,b)$ to be simple. When it is
not simple we determine all it submodules. In Sect.5 we determine
necessary and sufficient conditions for two such simple Virasoro
modules to be isomorphic.  In Sect.6 we show that the simple
Virasoro weight modules are new.
The good presentation of the modules $\mathcal{L}(W,\lambda,a,b)$
and the powerful technique in Proposition \ref{Key-computation}
enable us to establish the results in this paper.

 \section{Simple modules in
 $\mathcal{O}_{\mathfrak{W}}$}

In this section, we determine all simple modules in
 $\mathcal{O}_{\mathfrak{W}}$.

Let $V$ be a module over a Lie algebra $L$. We say that  the $L$-module $V$ is  trivial if $L V=0$. Denote by $\Soc_{L}(V)$ the socle
of the $L$-module $V$, i.e., $\Soc_{L}(V)$ is the sum of the minimal
nonzero submodules of $V$.  For any $v\in V$, the annihilator of $v$
is defined as $\ann_{L}(v)=\{g\in L| g v=0\}$.

For any $\mathfrak{b}$-module $B$ and $0\ne v\in B$, define
$\ord_{\mathfrak{b}}(v)$, the order of $v$, to be  the minimal
nonnegative integer $r$ with $d_{r+i}v=0$ for all $i\ge 1$, or to be
$\infty$ if such $r$ doesn't exist. And $\ord_{\mathfrak{b}}(B)$,
the order of $B$, is defined to be the maximal order of all its
elements or $\infty$ if it doesn't exist.

\begin{lemma}\label{char-O(b)}Suppose that $B\in \mathcal{O}_{\mathfrak{b}}$ is
simple.\begin{enumerate}[$($a$).$]
 \item
$\ord_{\mathfrak{b}}(B)=\ord_{\mathfrak{b}}(v)<\infty,$ for all nonzero
$v\in B$.
 \item $\ord_{\mathfrak{b}}(B)=0$ if and only if $B$ is one dimensional.
 \item If $ B$ is nontrivial, then the action of $d_r$ on $B$ is
bijective, where $r=\ord_{\mathfrak{b}}(B)$.\end{enumerate}
Consequently, $B$ is  a simple $\mathfrak{a}_r$-module for some
$r\in\N$. \end{lemma}

\begin{proof}For any nonzero $v,v'\in B$, since $B$ is simple, there exists some $u\in U(\mathfrak{b})$, such that $v'=uv$.
It is straightforward to check that $d_i v'=d_iuv=0$ for all
$i>\ord_{\mathfrak{b}}(v)$. So $\ord_{\mathfrak{b}}(v)\ge
\ord_{\mathfrak{b}}(v')$. Similarly we have
$\ord_{\mathfrak{b}}(v')\ge \ord_{\mathfrak{b}}(v)$. Thus
$\ord_{\mathfrak{b}}(v')=\ord_{\mathfrak{b}}(v)$. So we have proved
(a). Part (b) is trivial.

Now suppose that $ B$ is nontrivial, and $r=\ord_{\mathfrak{b}}(B)$.
Consider the subspace $X=\{v\in B|d_r v=0\}$ which is a proper
subspace of $B$. Then $X$ and $d_r(B)$ are $\mathfrak{b}$-submodules
of $B$. Since $B$ is simple, and $d_rB\ne0$ we deduce that $X=0$ and
$d_r(B)=B$, i.e., $d_r$ is bijective. Part (c) follows.
\end{proof}

\begin{lemma}\label{char-O(w)}Let $B\in
\mathcal{O}_{\mathfrak{b}}$, $W,W_1\in\mathcal{O}_{\mathfrak{W}}$ be
nontrivial simple modules. \begin{enumerate} \item
$\Ind_{\mathfrak{b}}^{\mathfrak{W}}(B)$ is a simple module in
$\mathcal{O}_{\mathfrak{W}}$;
\item $\Soc_{\mathfrak{b}}(W)$ is a simple $\mathfrak{b}$-module, and
an essential
$\mathfrak{b}$-submodule of $W$,
i.e. the intersection of all nonzero $\mathfrak{b}$-submodules
of $V$;
\item $W\cong\Ind_{\mathfrak{b}}^{\mathfrak{W}}\Soc_{\mathfrak{b}}(W),
B= \Soc_{\mathfrak{b}}(\Ind_{\mathfrak{b}}^{\mathfrak{W}}B)$;
\item $W\cong W_1$ if and only if $\Soc_{\mathfrak{b}}(W)\cong
\Soc_{\mathfrak{b}}(W_1)$.\end{enumerate}
Consequently, $W$ is the induced module from a simple $\mathfrak{a}_r$-module for some $r\in\N$.\end{lemma}

\begin{proof}(1). Let $M$ be a nonzero submodule of $\Ind_{\mathfrak{b}}^{\mathfrak{W}}(B)=\C[d_{-1}]\otimes B$.
Choose $0\ne v=\sum_{i=0}^s d_{-1}^i\otimes v_i\in M$ with minimal
$s$, where $v_i\in B$. Denote $r=\ord_{\mathfrak{b}}(B)$. If $s>0$,
then $$0\ne d_{r+1} v\in -s(r+2)d_{-1}^{s-1}\otimes d_r
v_s+\sum_{i=0}^{s-2} d_{-1}^i \otimes B\subset M,$$ which
contradicts the minimality of $s$. So $s=0$, i.e., $v\in 1\otimes
B$. Therefore $M=\Ind_{\mathfrak{b}}^{\mathfrak{W}}(B)$, and
$\Ind_{\mathfrak{b}}^{\mathfrak{W}}(B)$ is simple.

(2). Fix some $0\ne w\in W$ with minimal $\ord_{\mathfrak{b}} w=r$.
Let $M=U({\mathfrak{b}})w$. Then $\ord_{\mathfrak{b}}M=r$, and
$W=\C[d_{-1}] M$. 
For any $v=\sum_{i=0}^s d_{-1}^i w_i\in W$ with $w_s\ne 0$ and
$w_i\in M$ for $i=0,\ldots,s$, we have
 \begin{equation}\label{essential}0\ne d_{r+s} v=(-1)^s(r+s+1)(r+s)\cdots(r+2)
 d_rw_s\in M,\end{equation} \begin{equation}d_{r+i} v=0,\forall
i>s.\end{equation} Thus $\ord_{\mathfrak{b}}(v)=r+s$. So $W\cong
\C[d_{-1}]\otimes M$ and $W\cong
\Ind_{\mathfrak{b}}^{\mathfrak{W}}M$. Since $W$ is a simple
$\mathfrak{W}$-module,  $M$ has to be simple as
$\mathfrak{b}$-module, and it is essential from (\ref{essential}).

Part (3) is an obvious consequence of (1) and (2). Part (4) follows
from (3). 
\end{proof}

We remark that  a classification for all simple modules over $
\mathfrak{a}_1$ was given in \cite{Bl}, while  a classification for
all simple modules over $ \mathfrak{a}_2$ was recently obtained in
\cite{MZ3}. The problem is open for all other $(r+1)$-dimensional
Lie algebras  $ \mathfrak{a}_r$ for $r> 2$. However various simple
modules over $\mathfrak{a}_r$ were given in \cite{MZ3}.

\
\begin{example} Consider some $r\in\N$ and set
 $$\mu=(\mu_{r+1},\mu_{r+2},\dots,
\mu_{2r})\in\mathbb{C}^{r}.$$  Define the one dimensional
 $\Vir^{(r)}$ module $\C $ with the action

 $$d_i 1=0,\forall i>2r,$$

 $$d_k 1=\mu_k,\forall k=r+1,\ldots, 2r.$$

 Then we have the induced module $W_{\mu}=\Ind_{\Vir^{(r)}}^{\mathfrak{W}}
 \C$, which is simple if and only if $\mu_{2r}\ne 0$ or
 $\mu_{2r-1}\ne 0$. See \cite{LGZ} or \cite{MZ3} for more details.
\end{example}

\section{Constructing new Virasoro modules}

In this section, we will introduce our new weight Virasoro modules
to be studied in this paper. Then we provide our main technique for
later use.

\subsection{Constructing new Virasoro modules}

In this subsection we will provide a method to construct new weight
Virasoro modules from  modules in $\mathcal{O}_{\mathfrak{W}}$.

\

Let $W\in \mathcal{O}_{\mathfrak{W}}$. Then $W$
 can be naturally regarded as a  module over $\C[[t]]\frac{d}{d t}$.
 Regard $\C[e^{\pm t}]\frac{d}{ d t}$ as a subalgebra of $\C[[t]]\frac{d}{d
 t}$. We will use the expression
 $$e^{m t}=\sum_{k=0}^{\infty}\frac{(m t)^k}{k!}\in\C[[t]], \,\forall\, m\in\Z.$$
 In particular, we have
 $$[e^{mt}\frac{d}{d t}, e^{nt}\frac{d}{d t}]=(n-m)e^{(m+n)t}\frac{d}{d t},$$
$$[e^{mt}\frac{d}{d t}, \frac{d}{d
t}]=-me^{mt}\frac{d}{d t}, $$
 and $$(e^{mt}\frac{d}{d t})( \frac{d}{d t})=(e^{mt}\frac{d}{d
t})(\frac{d}{d t}-m),\,\forall\, m,n\in\Z,$$ in $U(\mathfrak{V})$.
Now we can give the weight modules defined and studied in this
paper.

 \begin{lemma}\label{w[lambda]} For any $W\in \mathcal{O}_{\mathfrak{W}}$, $a, b\in\C$ and  $\lambda\in \C^*$,
 the vector space $\mathcal{L}(W,\lambda,a,b)=W\otimes \C[t,t^{-1}]$ becomes a Virasoro module
 with the action
\begin{equation}z\cdot (w\otimes t^j)=0,\end{equation}
\begin{equation}d_k\cdot (w\otimes t^j)=((\lambda^k e^{kt}-1)\frac{d}{d t}+a+kb+j)w\otimes t^{k+j},\end{equation} for all
$k,j\in\Z, w\in W$.
\end{lemma}

\begin{proof} 
We only verify that $[d_m,d_n]\cdot (v\otimes
t^j)=(d_md_n-d_nd_m)\cdot (v\otimes t^j)$
for all $m,n,j\in\Z$ and $v\in W$, while other relations are
obvious. We compute (for simplicity, we will denote $\mu=a+j$)
$$\aligned  &d_md_n\cdot (v\otimes t^j)\\
=&d_m\cdot ((\lambda^n e^{nt}-1)\frac{d}{d t}+\mu+nb)v\otimes
t^{j+n})\endaligned $$ $$\aligned =&(\lambda^m e^{mt}-1)\frac{d}{d
t}+\mu+mb+n)((\lambda^n e^{nt}-1)\frac{d}{d t}+\mu+nb)v\otimes
t^{j+m+n}\\=&(\lambda^{m+n}(e^{mt}\frac{d}{d t})(e^{nt}\frac{d}{d
t})-\lambda^m(e^{mt}\frac{d}{d t})(\frac{d}{d t}-\mu-nb)
\\ &-(\frac{d}{d t}-\mu-mb-n)\lambda^n(e^{nt}\frac{d}{d
t}))v\otimes t^{m+n+j}\\
&+(\frac{d}{d t}-\mu-mb-n)(\frac{d}{d t}-\mu-nb)v\otimes
t^{m+n+j}\endaligned$$
$$\aligned =&(\lambda^{m+n}(e^{mt}\frac{d}{d
t})(e^{nt}\frac{d}{d t})-\lambda^m(e^{mt}\frac{d}{d t})(\frac{d}{d
t}-\mu-nb)\\ &\hskip 1cm -\lambda^n(e^{nt}\frac{d}{d t})(\frac{d}{d
t}-\mu-mb))v\otimes t^{m+n+j}\\
&+(\frac{d}{d t}-\mu-mb)(\frac{d}{d t}-\mu-nb)-n(\frac{d}{d
t}-\mu)+n^2b)v\otimes t^{m+n+j}.
\endaligned$$ 
Using the above formula, we deduce that
$$\aligned &(d_md_n-d_nd_m)\cdot (v\otimes
t^j)\\=&(\lambda^{m+n}[e^{mt}\frac{d}{d t},e^{nt}\frac{d}{d
t}]-(n-m)(\frac{d}{d t}-\mu)+(n^2-m^2)b)v\otimes
t^{m+n+j}\\=&(n-m)(\lambda^{m+n}e^{(m+n)t}\frac{d}{d t}-\frac{d}{d
t}+a+j+(m+n)b)v\otimes t^{m+n+j}
\\=&(n-m)d_{m+n}\cdot (v\otimes
t^j)\\=& [d_m,d_n]\cdot (v\otimes t^j).\endaligned$$
\end{proof}

From (3.2) we know that, if $M$ is infinite dimensional, then the
module $\mathcal{L}(W,\lambda,a,b)$ is a weight Virasoro module with
infinite dimensional weight spaces $\mathcal {L}_{\a+n}=W\otimes
t^n$ where
$$\mathcal {L}_{\a+n}=\{v\in \mathcal{L}(W,\lambda,a,b)\,\,|\,\,d_0v=(\a+n)v\}.$$

 \begin{example} When   $W=\C v$ is the trivial
 $\mathfrak{W}$-module, the weight $\Vir$-module $\mathcal{L}(W,\lambda,a,b)=v \otimes
\C[t,t^{-1}]$ is determined by
\begin{equation} d_m v_n=(a+n+bm)v_{n+m},\end{equation}
where $v_n=v\otimes t^n$, which is exactly the module $A_{a,b}$ of
intermediate series (see \cite{KR}).
\end{example}

\subsection{Realizing Virasoro  modules $E_h(b,\gamma,p)$ defined in \cite{CM}}
Let $W$ be the Verma $\mathfrak{W}$-module with the highest weight
vector $w_0$ of highest weight $b'\in\C$, i.e., $d_0 w_0=b' w_0$ and
$d_iw_0=0$ for all $i>0$. For any $a, b\in \C, \lambda\in \C^*$, we
have the weight $\Vir$-module $\mathcal{L}(W,\lambda,a,b)=W \otimes
\C[t,t^{-1}]$ with the action

 $$\aligned &z=0, \\&
d_n\cdot (f(d_{-1}) w_0\otimes t^i)\\=&((\lambda^n
e^{nt}-1)\frac{d}{d t}+a+nb+i)f(d_{-1})w_0\otimes
t^{n+i}\\=&(\lambda^nf(d_{-1}-n)e^{nt}\frac{d}{d
t}-f(d_{-1})d_{-1}+(a+nb+i)f(d_{-1}))w_0\otimes t^{n+i}\\=&
(\lambda^nf(d_{-1}-n)(nb'+d_{-1})-f(d_{-1})d_{-1}+(a+nb+i)f(d_{-1}))w_0\otimes
t^{n+i}\endaligned$$ for all $n, i\in \Z$, $f(t)\in \C[t]$. In
particular,
$$\aligned &d_n\cdot (d_{-1}^k w_0 \otimes
t^i)\\=&(\lambda^n(d_{-1}-n)^k(nb'+d_{-1})-d_{-1}^{k+1}+(a+nb+i)d_{-1}^k)
w_0\otimes t^{i+n},\endaligned$$ for  all $i,n\in \Z$ and $k\in
\Z_+$. For $k\in\Z_+, i\in\Z$, if we denote
$T_i^k=(-1)^kd_{-1}^kw_0\otimes t^i$, then
$$\aligned \ &d_n\cdot T_i^k=\lambda^n(n-d_{-1})^k(nb'+d_{-1})w_0\otimes
t^{i+n}+T_{i+n}^{k+1}+(a+nb+i)T_{i+n}^k\\
=&\left(\lambda^n(nb'+d_{-1})\sum_{j=0}^k
  {k\choose j}n^{k-j}(-d_{-1})^j\right )w_0\otimes t^{i+n}\\ &\hskip 2cm +T_{i+n}^{k+1}+(a+nb+i)T_{i+n}^k
\\=&\lambda^nnb'\sum_{j=0}^{k-1}{k\choose j}n^{k-j}T_{i+n}^j-\lambda^n\sum_{j=0}^{k-2}
{k\choose j}n^{k-j}T_{i+n}^{j+1}\\ &\hskip 2cm +(1-\lambda^n)T_{i+n}^{k+1}+(a+nb+i+\lambda^nnb'-\lambda^nnk)T_{i+n}^k.\\
  \endaligned$$
Comparing this action with the one in Lemma 2.1 of \cite{CM}, we see
that $\mathcal{L}(W,\lambda,a,b)$ is exactly the module
$E_h(a,-b',b+b')$ defined in \cite{CM} with $\lambda=e^h$ and
$\mu=a+i$. The main results in \cite{CM} are discussions on some
sufficient conditions for $E_h(a,-b',b+b')$ to be simple.

\subsection{Some properties of the $\C[x]$-module ${\mathcal
M}(\Z,\C)$} We are going to end this section with an important
result of computations, which  will be  frequently used later. This
technique is crucial to this paper.

 Let $P$ be any vector
space over $\C$. Denote by ${\mathcal M}(\Z,P)$ the set of all maps
from $\Z$ to $P$, which  naturally becomes a vector space over $\C$.

It is easy to see that ${\mathcal M}(\Z,P)$ becomes a module over
the polynomial algebra $\C[x]$ by the action
\begin{equation}\label{act-c[x]}( x^i \cdot T)(m)=T(m+i), \forall m\in \Z, T\in
{\mathcal M}(\Z,P).\end{equation}

Now we consider the infinite dimensional vector space ${\mathcal
M}(\Z,\C)$. For any $\lambda\in\C^*$ and $k\in \N$,  $\lambda^m m^k$
is regarded as an element in ${\mathcal M}(\Z,\C)$ by mapping $m\in
\Z$ to $\lambda^m m^k$. In particular, $x\cdot \lambda^m
m^k=\lambda^{m+1} (m+1)^k$.

 \begin{lemma}\label{Z-C} Let $\lambda\in \C^*$,  $k\in \N$, $p(x)\in \C[x]$ and $\lambda^m
 m^k\in {\mathcal M}(\Z,\C)$.

 (a) We have $$\aligned &(x-\lambda)^k \cdot (\lambda^m m^k)
 =k!\lambda^{m+k},\\ &(x-\lambda)^{k+1+i} \cdot (\lambda^m
 m^k)=0,\,\,\forall\,\,i\in \Z_+;\endaligned$$

 (b) $(p(x)(x-\lambda)^k)\cdot (\lambda^m
 m^k)=k!\lambda^{m+k}p(\lambda)$.

 (c) $p(x)\cdot (\lambda^m m^k)=0\in {\mathcal M}(\Z,\C)$ if and only if $(x-\lambda)^{k+1}|
 f(x)$.
 \end{lemma}

 \begin{proof} (a). We compute $$(x-\lambda)^k\cdot (\lambda^m m^k)=(x-\lambda)^{k-1}(x-\lambda)\cdot (\lambda^m
 m^k)$$
 $$=(x-\lambda)^{k-1}(x\cdot (\lambda^m
 m^k)-\lambda \lambda^m
 m^k)$$
$$=(x-\lambda)^{k-1}(\lambda^{m+1}
 (m+1)^k-\lambda^{m+1}
 m^k)$$
$$=(x-\lambda)^{k-1}(\lambda^{m+1}(k
 m^{k-1}+{\text{ lower terms of $m$}}))$$
 $$=...=k!\lambda^{m+k}.$$
 The second formula follows after applying another $(x-\lambda)$.

Part (b) follows from the fact that
$p(x)=q(x)(x-\lambda)+p(\lambda)$ for some $q(x)\in\C[x]$ by using
(a).

 (c). Note that $\ann_{\C[x]}(\lambda^m m^k)=\{g(x)\in \C[x]| p(x)\cdot (\lambda^m m^k)=0\}$
 is an ideal of the principle ideal domain $\C[x]$. From (a) we have $(x-\lambda)^k
 \notin \ann_{\C[x]}(\lambda^m m^k)$ and $(x-\lambda)^{k+1}
 \in \ann_{\C[x]}(\lambda^m m^k)$. Thus $\ann_{\C[x]}(\lambda^m m^k)
 =\C[x](x-\lambda)^{k+1}$. So we have proved (c).\end{proof}

Now we are ready to provide our main tool for later use.

\begin{proposition}\label{Key-computation}Let $P$ be a vector space over $\C$, $P_1$ be a subspace of $P$.
Suppose that $\lambda_1,\lambda_2,\ldots,\lambda_s\in \C^*$ are
pairwise distinct, $v_{i,j}\in P$ and $f_{i,j}(t)\in \C[t]$  with
$\deg(f_{i,j}(t))=j$ for $i=1,2,\ldots,s; j=0,1,2,\ldots,k.$ If
\begin{equation}\label{T} \sum_{i=1}^s \sum_{j=0}^k \lambda_i^m
f_{i,j}(m) v_{i,j}\in P_1, \forall m\in \Z,\end{equation} then $
v_{i, j}\in P_1$ for all $i, j$.

\end{proposition}

\begin{proof} Let $p(x)=((x-\lambda_1)(x-\lambda_2)\cdots
(x-\lambda_s))^{k+1}$, $q_j(x)=p(x)/(x-\lambda_j)^{k+1}$, and
$p_j(x)=p(x)/(x-\lambda_j)$ for $j=1,2,\ldots, s$. Denote
$$T_k(m)=\sum_{i=1}^s \sum_{j=0}^{k} \lambda_i^m f_{i,j}(m)
v_{i,j}\in P_1, \forall m\in \Z.$$ From (\ref{T}) we have $T_k\in
{\mathcal M}(\Z,P_1)$. By Lemma \ref{Z-C}(b), we see that
$$p_i(x)\cdot(\lambda_l^m f_{l,j}(m)
v_{l,j})=\delta_{i,l}\delta_{j,k}a_{i,k}q_i(\lambda_i)k!\lambda_i^{m+k}v_{i,k}.$$
where $a_{i,k}$ is the coefficient of  $t^k$ in $f_{i,k}$. So
\begin{equation}(p_i(x) \cdot T_k)(m)=a_{i,k}q_i(\lambda_i)k!\lambda_i^{m+k}v_{i,k}\in P_1, \forall\,\, m\in
\Z.
\end{equation}
Since $q_i(\lambda_i)\ne0$, we see that
 $v_{i,k}\in P_1$ for all $i=1,2,\ldots,s$, and
$$T_{k-1}(m)=\sum_{i=1}^s \sum_{j=0}^{k-1} \lambda_i^m
f_{i,j}(m) v_{i,j}\in P_1, \forall m\in \Z.$$ In this manner we
deduce that each $v_{i,j}\in P_1$.
\end{proof}

\begin{remark} In (\ref{T}), to satisfy the conditions for $f_{i,j}$, many $v_{i,j}$ may be zero.
If we take $P_1=0$, the corresponding result in this proposition becomes that $v_{i,j}=0$ for all $i,j$.\end{remark}

\section{Simplicity of Weight  Virasoro modules $\mathcal {L}(W,
\lambda,a,b)$}

For any $W\in \mathcal{O}_{\mathfrak{W}}$, $\lambda\in \C^*$, $a,
b\in\C$,  we have defined the  weight Virasoro module
$\mathcal{L}(W,\lambda,a,b)$ in Sect.3. In this section we will
determine   necessary and sufficient conditions for $\mathcal {L}(W,
\lambda,a,b)$ to be simple, and find all its submodules if it is not
simple.

For any simple $\mathfrak{b}$-module $B$  and any $b\in \C$, we can
have a new $\mathfrak{b}$-module structure on $B$, denoted by
$B_{(b)}$, with the new action $d_0 \cdot v=(d_0+b) v$ and $d_i
\cdot v= d_i v$ for all $v\in B$ and $i>0$. The new
$\mathfrak{b}$-module $B_{(b)}$ is also simple.

Let $W\in \mathcal{O}_{\mathfrak{W}}$ be simple.  We know that
$\Soc_{\mathfrak{b}}{W}$ is a simple $\mathfrak{b}$-module. Then
$(\Soc_{\mathfrak{b}}{W})\otimes \C[t,t^{-1}]$ is a submodule of
$\mathcal {L}(W, 1,a,0)$, which is exactly the Virasoro module
$\mathcal{N}(\Soc_{\mathfrak{b}}(W), a)$ defined and studied in
\cite{LLZ}. For $n\in \N$, in $W$ we define the subspace
$W^{(n)}=\sum_{i=0}^n d_{-1}^i (\Soc_{\mathfrak{b}}{W})$. The
following lemma solves the simplicity of $\mathcal {L}(W,
\lambda,a,b)$ for $\lambda=1$.

\begin{lemma}\label{lambda=1}For $a, b\in \C$ and
any nontrivial simple $W\in \mathcal{O}_{\mathfrak{W}}$, the
Virasoro module $\mathcal {L}(W, 1,a,b)$ has a filtration of
submodules
$$W^{(0)}\otimes \C[t,t^{-1}]\subset W^{(1)}\otimes
\C[t,t^{-1}]\subset \cdots \subset W^{(n)}\otimes
\C[t,t^{-1}]\subset \cdots,$$  with $(W^{(n)}\otimes
\C[t,t^{-1}])/(W^{(n-1)}\otimes \C[t,t^{-1}])\cong
\mathcal{N}((\Soc_{\mathfrak{b}}{W})_{(b)},a)$ for all $n\in \Z_+$.
\end{lemma}

\begin{proof} Since $W$ is nontrivial, from Lemma \ref{char-O(w)} we know that $W\cong \newline\Ind_{\mathfrak{b}}^{\mathfrak{W}}(\Soc_{\mathfrak{b}}W)$.
 For any $w\in \Soc_{\mathfrak{b}}{W}$, $m, k\in\Z$,
$n\in\Z_+$, we compute
$$d_m\cdot((d_{-1}^{n}w)\otimes t^k)
=((e^{mt}-1)\frac d{dt}+bm+a+k) (d_{-1}^{n}w)\otimes t^{k+m} $$
$$=((d_{-1}-m)^{n}(e^{mt}\frac d{dt})+d_{-1}^{n}(-d_{-1}+bm+a+k)) w\otimes t^{k+m} $$
$$\equiv d_{-1}^{n}(m(d_0+b)+\sum_{j=2}^\infty \frac {m^j}{j!}d_{j-1} +a+k) w\otimes
t^{k+m} \,\,\,$$ $$(\hskip -.5cm\mod   W^{(n-1)}\otimes
\C[t,t^{-1}]).
$$
We see that each $ W^{(n)}\otimes \C[t,t^{-1}]$ is a submodule of
$\mathcal {L}(W, 1,a,b)$ and   $(W^{(n)}\otimes
\C[t,t^{-1}])/(W^{(n-1)}\otimes \C[t,t^{-1}])\cong
\mathcal{N}((\Soc_{\mathfrak{b}}{W})_{(b)},a)$.
\end{proof}

From Theorem 4 in \cite{LLZ} we know that the above Virasoro module
$\mathcal{N}((\Soc_{\mathfrak{b}}{W})_{(b)},a)$ is simple if $W$ is
not a highest weight module.

The following lemma solves the simplicity of $\mathcal {L}(W,
\lambda,a,b)$  for $b=1$.

\begin{lemma}\label{b=1}For any nontrivial simple $W\in
\mathcal{O}_{\mathfrak{W}}$, the subspace
$$\mathcal {L}'(W, \lambda,a,1)=\oplus_{n\in \Z}((d_{-1}-a-n)W)\otimes t^n \subset
\mathcal {L}(W, \lambda,a,1)$$ is a submodule which is isomorphic to
$\mathcal {L}(W, \lambda,a,0)$ with the quotient $\mathcal {L}(W,
\lambda,a,1)/\mathcal {L}'(W, \lambda,a,1)\cong
\mathcal{N}((\Soc_{\mathfrak{b}}{W})_{(1)}, a).$
\end{lemma}

\begin{proof} Since $W$ is nontrivial, from Lemma \ref{char-O(w)} we know that $W\cong \newline\Ind_{\mathfrak{b}}^{\mathfrak{W}}(\Soc_{\mathfrak{b}}W)$.
Define  the linear map $\tau:\mathcal {L}(W, \lambda,a,0)\rightarrow
L'(W,\lambda,a,1)$ by $$\tau (w\otimes t^n)=((d_{-1}-a-n)w)\otimes
t^n, \forall w\in W, n\in \Z$$ which is clearly   bijective. Now for
any $m, n\in \Z$ and $w\in W$, we compute that
$$ d_m\cdot (((d_{-1}-a-n)w)\otimes t^n)$$ $$=(\lambda^me^{mt}\frac{d}{d
t}-\frac{d}{d t}+a+m+n)(d_{-1}-a-n)w\otimes t^{m+n}$$
$$=(\frac{d}{d t}-a-m-n)(\lambda^me^{mt}\frac{d}{d t}-\frac{d}{d
t}+a+n)w\otimes t^{m+n}\in \mathcal {L}'(W, \lambda,a,1),$$ to see
that $\mathcal {L}'(W, \lambda,a,1)$ is a submodule.

From the computations
$$\tau( d_m\cdot (w\otimes t^n))$$ $$=\tau((\lambda^me^{mt}\frac{d}{d
t}-\frac{d}{d t}+a+n)w\otimes t^{m+n})$$
$$=(\frac{d}{d t}-a-m-n)(\lambda^me^{mt}\frac{d}{d t}-\frac{d}{d
t}+a+n)w\otimes t^{m+n}\in \mathcal {L}'(W, \lambda,a,1),$$
$$d_m\cdot \tau(w\otimes t^n)=(\lambda^me^{mt}\frac{d}{d
t}+(-\frac{d}{d t}+a+n+m))(\frac{d}{d t}-a-n)w\otimes t^{m+n}$$
$$=((\frac{d}{d t}-a-n-m)\lambda^me^{mt}\frac{d}{d t}+(-\frac{d}{d
t}+a+n+m)(\frac{d}{d t}-a-n))w\otimes t^{m+n}$$ $$=(\frac{d}{d
t}-a-n-m)(\lambda^me^{mt}\frac{d}{d t}-\frac{d}{d t}+a+n)w\otimes
t^{m+n} \in \mathcal {L}'(W, \lambda,a,1),$$ we obtain that $d_m
\cdot \tau(w\otimes t^n)=\tau(d_m \cdot( w\otimes t^n))$. So  $\tau$
is a $\Vir$-module isomorphism.

From $$d_m\cdot (w\otimes \lambda^nt^n)=\lambda^n(\lambda^m
e^{mt}\frac{d}{d t}-\frac{d}{d t}+a+m+n)w\otimes t^{m+n}$$ $$\equiv
\lambda^{m+n} e^{mt}\frac{d}{d t}w\otimes t^{m+n} {\rm mod} \mathcal
{L}'(W, \lambda,a,1)$$
$$=((e^{mt}\frac{d}{d t}-\frac{d}{d t})+a+m+n)w\otimes
\lambda^{m+n}t^{m+n} {\rm mod} \mathcal {L}'(W, \lambda,a,1),$$ we
see that $\mathcal {L}(W, \lambda,a,1)/\mathcal {L}'(W, \lambda,a,1)
\cong \mathcal{N}((\Soc_{\mathfrak{b}}{W})_{(1)}, a)$.
\end{proof}

The following result shows that the Virasoro module $\mathcal
{L}(W,-1,a,b)$ is not simple if $W\in \mathcal {O}_{\mathfrak{W}}$
is a nontrivial highest weight module with highest weight $b-1$.

\begin{lemma}\label{b=b'+1}
 Let $W\in \mathcal{O}_{\mathfrak{W}}$ be a highest weight module with
the highest weight vector $w_0$ of the highest weight $b'\ne 0$.
Then
\begin{equation}\aligned &\mathcal{L}_0(a,b')=\span\{ d_1^i\cdot
(w_0\otimes t^{2j})|i\in\Z_+,j\in \Z\},\\
&\mathcal {L}_1(a,b')=\span\{ d_1^i\cdot (w_0\otimes
t^{2j+1})|i\in\Z_+,j\in \Z\}\endaligned\end{equation} are submodules
of $\mathcal {L}(W,-1,a,b'+1)$, and $$\mathcal
{L}(W,-1,a,b'+1)=\mathcal {L}_0(a,b')\oplus \mathcal {L}_1(a,b').$$
\end{lemma}

\begin{proof} By induction on $i\in\Z_+$ we can easily show that
 $$d_md_1^i\in \sum_{j=0}^{i} \C[d_1]d_{m+j}.$$ In $\mathcal {L}(W,-1,a,b'+1)$, for $ k,j\in \Z$ we
 compute
$$d_{2k} \cdot (w_0\otimes t^{2j})=(2k b'+a+2j+2k(b'+1))w_0\otimes
t^{2j+2k},$$ $$d_{2k+1} \cdot (w_0\otimes
t^{2j})=(-2d_{-1}+a+2k+2j+1)w_0\otimes t^{2k+2j+1}$$ $$=d_1\cdot
(w_0\otimes t^{2k+2j})\in \mathcal{L}_0(a,b'),$$ $$d_{2k} \cdot
(w_0\otimes t^{2j+1})=(2k b'+a+2j+1+2k(b'+1))w_0\otimes
t^{2j+2k+1},$$
$$d_{2k+1} \cdot (w_0\otimes t^{2j+1})=(-2d_{-1}+a+2k+2j+2)w_0\otimes
t^{2k+2j+2}$$ $$=d_1\cdot (w_0\otimes t^{2k+2j+1})\in
\mathcal{L}_1(a,b').$$ Using the above formulas we deduce that
$\mathcal {L}_0(a,b'), \mathcal {L}_1(a,b')$ are submodules of
$\mathcal {L}(W,-1,a,b'+1)$.

Since $b'\ne0$ we know that $W=\C[d_{-1}] w_0 =\C[d_{-1}]\otimes \C
w_0 $.

Denote $d_1^i\cdot (w_0\otimes t^j)=f_{i,j}(d_{-1})w_0\otimes
t^{i+j}$.
 Together with
$$d_1\cdot (f(d_{-1})w_0\otimes t^k)=(-2d_{-1}+a+k+b\hskip 3cm$$ $$+\sum_{j=1}^\infty \frac{
d_{j-1}}{j!})f(d_{-1})w_0\otimes t^{k+1}, \,\,\forall\,\,f(t)\in
\C[t], m,k\in\Z,$$ we may deduce  inductively that $f_{i,j}$ are
polynomials of degree $i$. We see  that  $\{ d_1^i\cdot (w_0\otimes
t^{j})|i\in\Z_+,j\in \Z\}$ is linearly independent. Consequently,
 $\{ d_1^i\cdot (w_0\otimes
t^{j})|i\in\Z_+,j\in \Z\}$ is  a basis for $\mathcal
{L}(W,-1,a,b'+1)$. Therefore we have $\mathcal
{L}(W,-1,a,b'+1)=\mathcal {L}_0(a,b')\oplus \mathcal {L}_1(a,b').$
 Thus we have proved the lemma.
\end{proof}

\begin{lemma}\label{lemma-19}Let $\lambda,a,b\in \C$ with $\lambda\ne 0,1$, and $W\in \mathcal{O}_{\mathfrak{W}}$ be
nontrivial and simple. Let $M$ be any submodule of $\mathcal
{L}(W,\lambda,a,b)$. If the subspace $\hat{M}=\{v\in W|v\otimes
\C[t,t^{-1}]\subseteq M\}$ is nonzero, then $M=\mathcal
{L}(W,\lambda,a,b)$.
\end{lemma}

\begin{proof} For any $v\in \hat{M}$, $l\in\Z$, we know that
$$d_{m}\cdot (v\otimes t^{l-m}) =(\sum_{i=0}^{\infty} \lambda^m
  m^{i}\frac{d_{i-1} v}{i!})\otimes t^l+(a+l-d_{-1}) v\otimes t^l$$ $$+ m(b-1) v\otimes
  t^l\in
  M, \forall m\in \Z.$$ From Proposition \ref{Key-computation},
  we have $d_i v \otimes t^l\in M$ for all $i\ge -1$ and $l\in \Z$,
  i.e., $d_i v\in \hat{M}$ for all $i\ge -1$. So
  $\hat{M}$ is a nonzero submodule of the simple $\mathfrak{W}$-module $W$, which has to be
  $W$. Therefore $M=\mathcal
{L}(W,\lambda,a,b)$.
 \end{proof}

Now we are ready to determine   necessary and sufficient conditions
for the weight Virasoro module $\mathcal {L}(W, \lambda,a,b)$ to be
simple.

\begin{theorem}\label{weight-simplicity}Let $\lambda\in \C^*, a, b\in
\C$, and $W\in \mathcal{O}_{\mathfrak{W}}$ be nontrivial simple.
\begin{enumerate}[$($a$).$] \item If $W$ is a highest weight module  with highest weight
$b'$, then the Virasoro module $\mathcal {L}(W,\lambda,a,b)$ is
simple if and only if  $\lambda\ne 1, -1$ and $b\ne 1$; or
$\lambda=-1$, $b\ne 1$ and $b\ne b'+1$.
\item If $W$ is not a highest weight module, then $\mathcal
{L}(W,\lambda,a,b)$ is simple if and only if $\lambda\ne 1$ and
$b\ne 1$.
\item If $W$ is a highest weight module with highest weight $b'\ne 0$,
then
\newline $\mathcal {L}_0(a,b')$ and
$\mathcal {L}_1(a,b')$ defined in Lemma \ref{b=b'+1} are the only
two nontrivial submodules of $\mathcal {L}(W,-1,a,b'+1) $, which are
simple.\end{enumerate}
\end{theorem}

\begin{proof} In $\mathcal {L}(W,\lambda,a,b)$, for any $w\in W, l, j, m\in \Z$, we have
\begin{equation}\label{4.2}\aligned &d_{l-m}d_m \cdot (w\otimes t^j) \\=&((\lambda^{l-m} e^{(l-m)t}\frac{d}{d t}-\frac{d}{d
t}+a+(l-m)b+m+j)\\&(\lambda^m e^{mt}\frac{d}{d t}-\frac{d}{d
t}+a+mb+j))w\otimes t^{j+l}\\=&(\lambda^l(e^{(l-m)t}\frac{d}{d
t})(e^{mt}\frac{d}{d t})+(-\frac{d}{d
t}+a+(l-m)b+m+j)\\&(-\frac{d}{d t}+a+mb+j))w \otimes
t^{l+j}\\&+\lambda^{-m}\lambda^{l} (-\frac{d}{d
t}+a+m(b-1)+j+l)(e^{(l-m)t}\frac{d}{d t})w \otimes
t^{l+j}\\&+\lambda^m(-\frac{d}{d t}+a+lb+(1-b)m+j) (e^{mt}\frac{d}{d
t})w \otimes t^{l+j}.
\endaligned\end{equation}
Let $M$ be a nonzero submodule of $\mathcal {L}(W,\lambda,a,b)$.
Take a nonzero
$$v=(\sum_{i=1}^s d_{-1}^i w_i)\otimes t^{i_0}\in M$$ with $w_s\ne 0$, $w_i\in
\Soc_{\mathfrak{b}}(W)$, and
$r=\ord_{\mathfrak{b}}(\Soc_{\mathfrak{b}}(W))\ge0$. From Lemma
\ref{char-O(b)} (c) we know that $d_r$ is bijective on the simple
$\mathfrak{b}$ module $\Soc_{\mathfrak{b}}(W)$.

Note that $$(e^{mt}\frac d{dt})(d_{-1}^i
w_i)=(d_{-1}-m)^i(e^{mt}\frac d{dt}) w_i,$$ $$(e^{(l-m)t}\frac{d}{d
t})(e^{mt}\frac d{dt})(d_{-1}^i
w_i)=(d_{-1}-l)^i(e^{(l-m)t}\frac{d}{d t})(e^{mt}\frac d{dt}) w_i,$$
and $d_{r+k} w_i=0$ for all $k\in\N$.  From (\ref{4.2}), we may
write
\begin{equation}\label{expand}\aligned d_{l-m}d_m \cdot (v)&=\sum_{i=0}^{2r+2} m^i v_{1,i}^{(l,i_0)}\otimes t^{l+i_0}\\
&+\sum_{i=0}^{r+s+2} \lambda^{-m} m^i v_{\frac{1}{\lambda},
i}^{(l,i_0)}\otimes t^{l+i_0}\\&+\sum_{i=0}^{r+s+2}\lambda^m m^i
v_{\lambda,i}^{(l,i_0)}\otimes t^{l+i_0}\in
M,\endaligned\end{equation} where
$v_{1,i}^{(l,i_0)},v_{\frac{1}{\lambda},
i}^{(l,i_0)},v_{\lambda,i}^{(l,i_0)}\in W$ are independent  of $m$.
In particular, 
\begin{equation}\label{highest 1}v_{\lambda,
r+s+2}^{(l,i_0)}=\frac{(1-b)(-1)^s}{(r+1)!} d_r w_s, \forall l\in
\Z,\end{equation}
\begin{equation}\label{45}v_{\frac1{\lambda},
r+s+2}^{(l,i_0)}=\frac{(-1)^{r+s}(1-b)\lambda^l}{(r+1)!} d_r w_s,
\forall l\in \Z.\end{equation} And if $r>0$,
\begin{equation}\label{47}v_{1, 2r+s+2}^{(l,i_0)}=\frac{(-1)^{r+1}
\lambda^l}{((r+1)!)^2}(d_{-1}-l)^sd_r^2 w_s, \forall l\in
\Z.\end{equation}

\begin{case} $\lambda\ne 1, -1$ and $b\ne 1$.\end{case}

From (\ref{expand}), (\ref{highest 1}) and Proposition
\ref{Key-computation}, we have
$$d_rw_s\otimes t^{l+i_0}\in  M, \forall l\in \Z.$$

From Lemma \ref{lemma-19}, we have $M=\mathcal {L}(W,\lambda,a,b)$.
Thus $\mathcal {L}(W,\lambda,a,b)$ is simple in this case.

\begin{case} $\lambda=-1$ and $b\ne 1$.\end{case}

From (\ref{highest 1}) and (\ref{45}) we see that
\begin{equation}\label{highest2}
v_{-1,r+s+2}^{(l,i_0)}+v_{\frac 1{-1},r+s+2}^{(l,i_0)}=\frac{(1-b)(-1)^s(1+(-1)^{l+r})}{(r+1)!}d_rw_s.\end{equation}

Thus from (\ref{expand}) and Proposition \ref{Key-computation} we
have
\begin{equation}\label{4.5} d_rw_s\otimes t^{r+s+2k+i_0}\in M,
\forall k\in \Z.\end{equation}

  Now replacing $v$ with $d_rw_s\otimes
t^{-r+i_0}$ if necessary, we may assume that
\begin{equation}\label{4.6}0\ne v=w\otimes t^{i_0+2k}\in
M, \,\,\forall\,\, k\in \Z\end{equation} where $w\in
\Soc_{\mathfrak{b}}(W).$ Now $s=0$ in
(\ref{expand})-(\ref{highest2}).

\begin{subcase}$W$ is not a highest weight module. \end{subcase}

 Note that in
this case we have $r>0$. Then (4.6) becomes 
$$v_{1,2r+2}^{(l,i_0)}
=\frac{(-1)^{l+r+1}}{((r+1)!)^2}d_r^2w.$$
So from Proposition \ref{Key-computation}, we have $$ (d_r^2w)\otimes
t^{l+i_0}\in  M, \forall l\in \Z.$$
 From Lemma \ref{lemma-19}, we
have $M=\mathcal {L}(W,\lambda,a,b)$. Thus $\mathcal
{L}(W,\lambda,a,b)$ is simple in this case.

\begin{subcase}$W$ is a highest weight module with highest weight $b'\ne 0$, and $b=b'+1$.\end{subcase}

In this case, we have $r=s=0$. From (\ref{4.6}) and the proof of
Lemma \ref{b=b'+1}, we have either $\mathcal {L}_0(a,b')\subset M$
if $i_0$ is even or $\mathcal {L}_1(a,b')\subset M$ if
$i_0$ is odd. 
Combining with Lemma \ref{b=b'+1} we see that $\mathcal {L}_0(a,b')$
and $\mathcal {L}_1(a,b')$ are the only two simple Virasoro
submodules in $\mathcal {L}(W,\lambda,a,b)$. And we also see that
${L}_0(a,b')$ and ${L}_1(a,b')$ are not isomorphic.

\begin{subcase}$W$ is a highest weight module with highest weight $b'\ne 0$, and $b\ne b'+1$.\end{subcase}

For this case,  we have $r=s=0$. In (\ref{4.6}),  $w$ is the highest
weight vector of $W$. For any $m, l\in\Z$ we have
 \begin{equation}\label{case2.3}\aligned d_{2m+1}\cdot& (w\otimes
  t^{i_0-2m+2l})\\&=2m(b-b'-1)w\otimes t^{i_0+2l+1}+\\ &\hskip 1cm (b-b'-2d_{-1}+a+i_0+2l)w\otimes
  t^{i_0+2l+1}\in M.\endaligned\end{equation} Since $b\ne b'+1$. we have $w\otimes
  t^{i_0+2l+1}\in M$ for all $l\in \Z$. Thus $w\otimes
  \C[t,t^{-1}]\subset M$. From Lemma \ref{lemma-19}, we
have $M=\mathcal {L}(W,\lambda,a,b)$. So $\mathcal
{L}(W,\lambda,a,b)$ is simple in this case.

We know that  $\mathcal {L}(W,\lambda,a,b)$ is not simple   if
$\lambda=1$ (Lemma \ref{lambda=1}), or if $b=1$ (Lemma \ref{b=1}).
So we have completed the proof.
\end{proof}

Note that, if $\C$ is the trivial $\mathfrak{W}$-module, then
$A_{a,b}=\mathcal {L}(\C,1,a,b)$ is the module of intermediate
series. See \cite{KR} and Example 2.

Applying Lemmas \ref{lambda=1}, \ref{b=1}, \ref{b=b'+1}, and Theorem
\ref{weight-simplicity} to  modules $E_h(b,\gamma, p)$ defined in
\cite{CM} (where we take the $\mathfrak{W}$-module $W$ as the Verma
module of highest weight $-\gamma$), we can completely determine the
structure of these modules.

\begin{corollary} Let $h, b, \gamma, p\in\C$.
\begin{enumerate}
 \item $E_h(b,\gamma, p)$ is simple if and only if $ e^h\ne 1,-1$, and $\gamma\ne 0, 1-p$; or $e^h=-1$, and $\gamma\ne 0, 1-p, \frac{1-p}{2}$.
 \item If $e^h=1$ and $\gamma\ne0$, then $E_h(b,\gamma,p)$ has a filtration
$$0=M^{(0)}\subset M^{(1)}\subset \cdots \subset M^{(n)}\subset
\cdots $$ with $M^{(i+1)}/M^{(i)}\cong A_{b,p}$ for all $i\in \Z_+$.
 \item If $e^h\ne 1$ and $\gamma\ne 0$, then $E_h(b,\gamma, 1-\gamma)$
has a proper submodule $M\cong E_h(b,\gamma,-\gamma)$ with
$E_h(b,\gamma, 1-\gamma)/M\cong A_{b, 1-\gamma}$.
 \item If $e^h=-1$, and $\gamma\ne 0$, then $E_h(b,\gamma,1-2\gamma)$
has exactly two simple submodules $E_{\pm}(b,\gamma,1-2\gamma)$, and
$$E_h(b,\gamma,1-2\gamma)=E_{+}(b,\gamma,1-2\gamma)\oplus
E_{-}(b,\gamma,1-2\gamma)$$ with $E_{+}(b,\gamma,1-2\gamma)\not
\cong E_{-}(b,\gamma,1-2\gamma)$.
 \item  If $e^h=-1$, then $E_h(b,0,1)$ has submodules $M_2\subset M_1$
with $M_1/M_2\cong A_{b,0}$, $E_h(b,0,1)/M_1\cong A_{b,1}$,
$M_1\cong E_h(b,1,0)$, $M_2\cong E_h(b,1,-1)\cong
E_{+}(b,1,-1)\oplus E_{-}(b,1,-1)$.
 \item  If $e^h=-1$, then $E_h(b,0,0)$ has a submodule  $M\cong
E_h(b,1,-1)$ and $E_h(b,0,0)/M\cong A_{b,0}$, and $M\cong
 E_{+}(b,1,-1)\oplus E_{-}(b,1,-1)$.
 \item Let $e^h\ne 1,-1$ and $p\ne 1$; or $e^h=-1$ and $p\ne 0,1$. Then
$E_h(b,0,p)$ has an simple submodule $M$ with $M\cong E_h(b,1,p-1)$
and $E_h(b,0,p)/M\cong A_{b,p}$.
 \item If $e^h\ne 1, -1$, then $E_h(b,0,1)$ has submodules $M_2\subset
M_1$ with $M_1\cong E_h(b,1,0)$, $M_2\cong E_h(b,1,-1)$,
$M_1/M_2\cong A_{b,0}$, and $E_h(b,0,1)/M_1\cong A_{b,1}$, where
$M_2$ is simple.
\end{enumerate}
\end{corollary}

\begin{proof} From  Sect.3.2, we know that if $\gamma\ne 0$, then
 $E_h(b,\gamma,p)\cong \newline\mathcal {L}(W,e^h,b,\gamma+p)$
where $W$ is the simple highest weight $\mathfrak{W}$-module with
highest weight $-\gamma$.

Part (1) follows directly from Theorem \ref{weight-simplicity}(a),
Part (2) follows directly from Lemma \ref{lambda=1}, Part (3) follows directly from Lemma \ref{b=1},
and Part (4) follows directly from Lemma \ref{b=b'+1}.

Now let  $\gamma=0$,  let $W$ be the Verma $\mathfrak{W}$-module
with highest weight $0$, and let $W'$ be the unique simple submodule
of $W$ with highest weight $-1$. (So $\gamma=1$ for $W'$). We have
$E_h(b,0,p)\cong \mathcal {L}(W,e^h,b,p)$. Note that in this case
$\mathcal {L}(W,e^h,b,p)$ has a submodule $$M\cong \mathcal{L}(W',
e^h,b,p)\cong E_h(b,1,p-1)$$ with $\mathcal {L}(W,e^h,b,p)/M\cong
A_{b,p}$.

(5). We see that $E_h(b,0,1)\cong \mathcal{L}(W, -1,b,1)$. Then $E_h(b,0,1)$ has submodule $M_1
\cong E_h(b,1,0)\cong \mathcal{L}(W', -1,b,1)$. Applying Lemmas \ref{b=1}, \ref{b=b'+1}, we see that $M_1$ has a submodule $M_2\cong E_h(b,1,-1)\cong
E_{+}(b,1,-1)\oplus E_{-}(b,1,-1)$. Also $M_1/M_2\cong A_{b,0}$.

(6). We see that $E_h(b,0,0)\cong \mathcal{L}(W, -1,b,0)$. Then $E_h(b,0,0)$ has submodule $M
\cong E_h(b,1,-1)\cong \mathcal{L}(W', -1,b,0)$. Applying Lemma \ref{b=b'+1}, we see that $M$ has  submodules $M\cong E_h(b,1,-1)\cong
E_{+}(b,1,-1)\oplus E_{-}(b,1,-1)$. Also $M_1/M_2\cong A_{b,0}$.

Part (7) follows from  Theorem \ref{weight-simplicity}(a), while Part (8) follows from  Lemma \ref{b=1}.
\end{proof}

\begin{example} Let $\mu=(\mu_1,\mu_2)\in \C^2, \lambda \in \C^*$. Let $W_{\mu}$ be as in Example 1. Then
$W_{\mu}$ is simple if and only if $\mu_1$ or $\mu_2$ is nonzero. We
can easily see that $W_{\mu}=\C[ d_{-1}]\otimes \C[d_0]$. For any
$\lambda, a, b\in \C$ with $b\ne 1$ and $\lambda\notin\{0,1\}$, we
obtain the simple weight Virasoro module
$V(\mu_1,\mu_2,\lambda,a,b)=\C[d_{-1},d_0]\otimes \C[t,t^{-1}]$ with
the action $$ z\cdot V(\mu_1,\mu_2,\lambda,a,b)=0,$$ $$d_m\cdot
(d_{-1}^id_0^j\otimes t^k)\hskip 5cm $$
$$=\lambda^m(d_{-1}-m)^i(d_{-1}d_0^j+md_0^{j+1}+\frac{m^2}{2}\mu_1(d_0-1)^j+\frac{m^3}{6}\mu_2(d_0-2)^j)\otimes
t^{k+m}$$
$$ +(-d_{-1}+a+k+bm)d_{-1}^id_0^j\otimes t^{k+m}$$ for all  $i,
j\in\Z_+$ and $ k, m\in \Z.$
\end{example}

\section{Isomorphism classes of Virasoro modules $\mathcal {L}(W,
\lambda,a,b)$}

In this section, we will determine the isomorphism classes between
the simple weight modules we have obtained from the weight Virasoro
module $\mathcal {L}(W, \lambda,a,b)$.

It is easy to see that $\mathcal{L}(W,\lambda, a,b)\cong
\mathcal{L}(W,\lambda,a+n,b)$ for all $n\in \Z$. Thus without lose
of generality, we may assume that $0\le {\rm Re} a<1$.

Let $W\in \mathcal{O}_{\mathfrak{W}}$ be a highest weight module
with the highest weight vector $w_0$ of the highest weight $b'\ne
0$. Then we know that $W=\C[d_{-1}]\otimes w_0$. For the simple
submodules $\mathcal{L}_0(a,b'), \mathcal{L}_1(a,b')$ of   $\mathcal
{L}(W, \lambda,a,b'+1)$  defined in Lemma \ref{b=b'+1}, it is not
hard to check that for all $n\in \Z$, $$\mathcal{L}_0(a,b')\cong
\mathcal{L}_0(a_1+2n,b'),$$
$$\mathcal{L}_1(a,b')\cong
\mathcal{L}_0(a+1,b').$$

The following lemma gives some isomorphisms between two different Virasoro modules $\mathcal {L}(W, \lambda,a,b)$.

\begin{lemma}\label{iso} Let $\lambda\in \C^*$ and $a, b', b_0'\in \C$. Let $W, W_0\in \mathcal{O}_{\mathfrak{W}}$ be the Verma modules with
 highest weight vectors $w, w_0$ of the highest weights $b', b'_0$ respectively.
Then the linear map $$\aligned \phi: \,\,\,&\mathcal {L}(W,
\lambda,a,b_0'+1)\to \mathcal {L}(W_0, \lambda^{-1},a,b'+1),\\ &
(d_{-1}^kw)\otimes t^l\mapsto \lambda^{-l}((a+l-d_{-1})^kw_0)\otimes
t^l,\endaligned$$ is a $\Vir$-module isomorphism.\end{lemma}

\begin{proof} We need only to verify that
$\phi(d_m\cdot (d_{-1}^kw\otimes t^{l}))=d_m\cdot
\phi(d_{-1}^kw\otimes t^{l})$ for all $m,l\in\Z$ and $k\in\Z_+$. We
compute
$$\aligned&d_m\cdot \phi(d_{-1}^kw\otimes t^{l})\\ =&d_m\cdot (\lambda^{-l}(a+l-d_{-1})^kw_0\otimes t^{l})\\
=&\lambda^{-l} (\lambda^{-m}e^{mt}\frac{d}{dt}-d_{-1}+a+l+(1+b')m) (a+l-d_{-1})^kw_0)\otimes t^{l+m}\\
=&((\lambda^{-l-m}(a+l+m-d_{-1}) ^k(d_{-1}+b_0'm)\\
&+\lambda^{-l} (a+l+(1+b')m-d_{-1}) (a+l-d_{-1})^k)w_0)\otimes
t^{l+m},
\endaligned$$

$$\aligned &\phi(d_m\cdot (d_{-1}^kw\otimes t^{l}))\\
=&\phi((\lambda^me^{mt}\frac{d}{dt}-d_{-1}+a+l+m+b_0'm)d_{-1}^kw\otimes
t^{l+m})
\\
=&\phi((\lambda^m(d_{-1}-m)^k(d_{-1}+b'm)+(a+l+m-d_{-1}+b_0'm)d_{-1}^k)w\otimes t^{l+m})\\
=&\lambda^{-l-m}(\lambda^m(a+l-d_{-1})^k(a+l+m-d_{-1}+b'm)w_0\otimes
t^{l+m}\\
&+\lambda^{-l-m}(d_{-1}+b_0'm)(a+l+m-d_{-1})^k)w_0\otimes t^{l+m}\\
=&((\lambda^{-l-m}(a+l+m-d_{-1}) ^k(d_{-1}+b_0'm)\\
&+\lambda^{-l} (a+l+(1+b')m-d_{-1}) (a+l-d_{-1})^k)w_0)\otimes
t^{l+m}.
\endaligned$$
This completes the proof.\end{proof}

\begin{remark} In \cite{CM}, it was proved that $E_h(b,\gamma,p)$ and $E_{-h}(b,1-\gamma-p,p)$ are equivalent. The lemma above shows that in fact we have
 $E_h(b,\gamma,p)\cong E_{-h}(b,1-\gamma-p,p)$.\end{remark}

Now we are going to prove the main result in this section.

\begin{theorem}Let $\lambda, \lambda_0 \in \C^*, a, b, b', b'_0, a_i, b_i\in \C$ for $i=0,1,2$ with
$0\le {\rm Re} a,{\rm Re}
a_i <1$ and
 $b'b'_0b_1b_2\ne 0$. Let  $B\in \mathcal{O}_{\mathfrak{b}}$ and $W, W_0\in \mathcal{O}_{\mathfrak{W}}$
 be nontrivial
simple modules.

\begin{enumerate}[$($a$).$] \item
 Suppose that  $\mathcal{L}(W,\lambda, a,b)$ is simple. Then
$$\mathcal{L}(W,\lambda, a,b)\cong \mathcal{L}(W_0,\lambda_0,
a_0,b_0)$$ if and only if one of the following holds
\begin{enumerate}[$($i$).$]
\item
$W\cong W_0$, $\lambda=\lambda_0$, $a=a_0$
and $b=b_0$, or
\item $W, W_0$ are highest weight modules with highest weights $b', b'_0$ respectively, and
$\lambda_0=\lambda^{-1}, a=a_0, b=b'_0+1, b_0=b'+1$.
\end{enumerate}  \item For $i=0,1$, we have
$\mathcal{L}_i(a_1,b_1)\cong \mathcal{L}_i(a_2,b_2)$ if and only if
$a_1=a_2$, and $b_1=b_2$.
 \item The modules $\mathcal{L}(W,\lambda, a,b), \mathcal{N}(B,a_0),
\mathcal{L}_0(a_1,b_1), \mathcal{L}_1(a_2,b_2)$ are pairwise
non-isomorphic.
\end{enumerate}
\end{theorem}

\begin{proof} (a). The sufficiency is trivial. Now suppose that $\phi:L(W,\lambda,
a,b)$ $\rightarrow L(W_0,\lambda_0, a_0,b_0)$ is a module
isomorphism. It is clear that $a=a_0$. Since $\mathcal{L}(W,\lambda,
a,b)$ is simple,  $\lambda, \lambda_0,   b,   b_0$ are different
from $1$. Denote $r=\ord_{\mathfrak{b}}(W)$,
$r'=\ord_{\mathfrak{b}}(W_0)$ and $\phi(v\otimes
t^n)=\phi_n(v)\otimes t^n$. Then each $\phi_n: W\to W_0$ is a vector
space isomorphism.

Let $M=\{v\in W| \phi_n(v)=\phi_0(v),\text{ for all } n\in \Z\}$. If
$M\ne0$,  for any $w\in M, l,m\in\Z$ we have
$$\phi_l((\lambda^m\sum_{j=0}^\infty\frac{m^j}{j!}d_{j-1}-d_{-1}+a+l-m+bm) w)\otimes
t^{l}$$
$$\aligned =&\phi(((\lambda^m\sum_{j=0}^\infty\frac{m^j}{j!}d_{j-1}-d_{-1}+a+l-m+bm)
w)\otimes t^{l})\\
=&\phi(d_m\cdot (w\otimes t^{l-m}))\\ =&d_m\cdot \phi(w\otimes
   t^{l-m})\\
   =&(\lambda_0^m\sum_{j=0}^\infty\frac{m^j}{j!}d_{j-1}-d_{-1}+a+l-m+b_0m)
   \phi_0(w)\otimes t^{l},\endaligned$$
   yielding that $$\lambda=\lambda_0, b=b_0,
   \phi_l(d_j w)=d_j\phi_0(w),\,\forall   l\in \Z, j\ge -1, w\in M.$$ Thus $M$ is a
   $\mathfrak{W}$-submodule of $W$. Then $M=W$ and $\phi_l=\phi_0$ is an
   $\mathfrak{W}-$module isomorphism. So (i) holds in this case. Thus we only need to prove that $M\ne 0$ or (ii).

\

For any $w\in\Soc_{\mathfrak{b}}(W), l,m,n\in\Z$, similar to
(\ref{4.2}) we have
  \begin{equation}\label{4.11}\aligned& (\lambda_0^l(e^{(l-m)t}\frac{d}{d
t})(e^{mt}\frac{d}{d t})+(-\frac{d}{d t}+a+(l-m)b_0+m+n)\\
&(-\frac{d}{d t}+a+mb_0+n))\phi_n(w) \otimes
t^{l+n}\\&+\lambda_0^{-m}\lambda_0^{l} (-\frac{d}{d
t}+a+m(b_0-1)+n+l)(e^{(l-m)t}\frac{d}{d t})\phi_n(w) \otimes
t^{l+n}\\
&+\lambda_0^m(-\frac{d}{d t}+a+lb_0+(1-b_0)m+n) (e^{mt}\frac{d}{d
t})\phi_n(w) \otimes t^{l+n}\\= & ((\lambda_0^{l-m}
e^{(l-m)t}\frac{d}{d t}-\frac{d}{d t}+a+(l-m)b_0+m+n)\\&(\lambda_0^m
e^{mt}\frac{d}{d t}-\frac{d}{d t}+a+mb_0+n)\phi_n(w))\otimes
t^{n+l}\\ =&d_{l-m}d_m\cdot (\phi_n(w)\otimes
   t^{n})\\ =&\phi(d_{l-m}d_m\cdot (w\otimes
   t^{n+l}))\\= & \phi_{n+l}(((\lambda^{l-m}
e^{(l-m)t}\frac{d}{d t}-\frac{d}{d t}+a+(l-m)b+m+n)\\&(\lambda^m
e^{mt}\frac{d}{d t}-\frac{d}{d t}+a+mb+n)w)\otimes
t^{n+l}\endaligned\end{equation}
$$\aligned
  =& \phi_{l+n}(((\lambda^l(e^{(l-m)t}\frac{d}{d
t})(e^{mt}\frac{d}{d t})+(-\frac{d}{d t}+a+(l-m)b+m+n)\\
&(-\frac{d}{d t}+a+mb+n))w )\otimes
t^{l+n}\\&+\lambda^{-m}\lambda^{l}\phi_{l+n}( (-\frac{d}{d
t}+a+m(b-1)+n+l)(e^{(l-m)t}\frac{d}{d t})w )\otimes
t^{l+n}\\
&+\lambda^m\phi_{l+n}((-\frac{d}{d t}+a+lb+(1-b)m+n)
(e^{mt}\frac{d}{d t})w )\otimes t^{l+n}.\endaligned$$  We can write
$\phi_n(w)=\sum_{i=1}^{s'} d_{-1}^i w'_i $ with  $w'_i\in
\Soc_{\mathfrak{b}}(W')$,   and $w'_{s'}\ne 0$.

   Similarly as we have done in (\ref{expand})-(4.6), we may write both sides
   of (\ref{4.11})  as a linear combination of
   $\{m^i, \lambda^{\pm i} m^i, \lambda_0^{\pm i}m^i\}.$ By
   comparing with the highest degree of $m^i$ we see
   that $r=r'$. Further comparing with the highest degree of
   $\lambda^{\pm m} m^i$, from  the analogues of (\ref{highest 1})-(4.6),
   we have $s'=0$, and $\lambda=\lambda_0$ or
   $\lambda=\lambda_0^{-1}$. In particular, we have
   $\phi(\Soc_{\mathfrak{b}}(W)\otimes \C[t,t^{-1}])\subset \Soc_{\mathfrak{b}}(W_0)\otimes
   \C[t,t^{-1}]$, i.e., $\phi_n(\Soc_{\mathfrak{b}}(W))=
   \Soc_{\mathfrak{b}}(W_0)$ for any $n\in\Z$.

Now we can write (5.1) as
\begin{equation}\label{expand2}
\sum_{i=0}^{2r+2} m^i v_{1,i}^{(l,n)}\otimes
t^{l+n}+\sum_{i=0}^{r+2} \lambda_0^{-m} m^i v_{\frac{1}{\lambda_0},
i}^{(l,n)}\otimes t^{l+n}+\sum_{i=0}^{r+2}\lambda_0^m m^i
v_{\lambda_0,i}^{(l,n)}\otimes t^{l+n}\end{equation}
$$\aligned
=&\sum_{i=0}^{2r+2} m^i \phi_{n+l}(w_{1,i}^{(l,n)})\otimes
t^{l+n}+\sum_{i=0}^{r+2} \lambda^{-m} m^i
\phi_{n+l}(w_{\frac{1}{\lambda}, i}^{(l,n)})\otimes t^{l+n}\\
&+\sum_{i=0}^{r+2}\lambda^m m^i
\phi_{n+l}(w_{\lambda,i}^{(l,n)})\otimes t^{l+n},
\endaligned$$
where $v_{1,i}^{(l,n)},v_{\frac{1}{\lambda_0},
i}^{(l,n)},v_{\lambda_0,i}^{(l,n)}\in W_0$,
$w_{1,i}^{(l,n)},w_{\frac{1}{\lambda},
i}^{(l,n)},w_{\lambda,i}^{(l,n)}\in W$ are independent  of $m$. In
particular,
\begin{equation}\label{3}v_{\lambda_0,
r+2}^{(l,n)}=\frac{(1-b_0)}{(r+1)!} d_r \phi_n(w), \forall l, n\in
\Z,\end{equation}
\begin{equation}\label{4}w_{\lambda,
r+2}^{(l,n)}=\frac{(1-b)}{(r+1)!} d_r w, \forall l, n\in
\Z,\end{equation}
\begin{equation}\label{5}v_{\frac1{\lambda_0},
r+2}^{(l,n)}=\frac{(-1)^r(1-b_0)\lambda_0^l}{(r+1)!} d_r \phi_n(w),
\forall l, n\in \Z,\end{equation}
\begin{equation}\label{6}w_{\frac1{\lambda},
r+2}^{(l,n)}=\frac{(-1)^r(1-b)\lambda^l}{(r+1)!} d_r w, \forall l,
n\in \Z.\end{equation}

And if $r>0$,
\begin{equation}\label{7}v_{1,
2r+2}^{(l,n)}=\frac{(-1)^{r+1} \lambda_0^l}{((r+1)!)^2}d_r^2
\phi_n(w), \forall l, n\in \Z,\end{equation}
\begin{equation}\label{8}w_{{1,
2r+2}}^{(l,n)}=\frac{(-1)^{r+1} \lambda^l}{((r+1)!)^2}d_r^2 w,
\forall l, n\in \Z.\end{equation}

 \

   {\bf Case 1}. $\lambda\ne -1$.

   If $W$ is not a highest weight module, i.e., $r\ge 1$,  by comparing with
   the coefficient of $m^{2r+2}$ in   (5.2) and using (5.7) and (5.8),
   we have $\phi_{l+n}(d_r^2
   w)=d_r^2\psi_{n}(w),$  for all $n,l\in \Z$, and $w\in
   \Soc_{\mathfrak{b}}(W)$. Thus $M\ne 0$ in this case.

   Now suppose that $W$ is a highest weight module. So $r=r'=0$.

   {\bf Subcase 1.1}. $\lambda=\lambda_0$

   In this case, from (5.3) and (5.4) we have
 $$(1-b_0)d_0 \phi_{n}(w)=(1-b)\phi_{l+n}(d_0 w),$$ for all $l,n\in \Z, w\in
   \Soc_{\mathfrak{b}}(W)$. Thus $M\ne 0$.

  {\bf Subcase 1.2}. $\lambda=\lambda_0^{-1}$

   In this case, by computing the coefficients of $\lambda^m m^2$ in  (5.2) and using (5.4) and (5.5),
    we have $\lambda^l(1-b)\phi_{l+n}(d_0 w)=(1-b_0) d_0
   \phi_{n}(w)$ for all $l\in\Z$, where $w$, $\phi_{n}(w)$ are the highest weight vectors in $W$ and $W_0$
   respectively. We may assume that $d_0w=b'w$ and $d_0(\phi_{n}(w))=b'_0\phi_{n}(w)$.
   Then $b'b'_0\ne0$ because of the simplicity of the modules.
   We see that
   $\lambda^l(1-b)b'\phi_{l+n}( w)=(1-b_0) b'_0
   \phi_{n}(w)$.
   By taking $l=0$ we deduce that $(1-b)b'=(1-b_0)b'_0\ne0$. Thus \begin{equation}\label{9}\phi_l(w)=\lambda^{-l}
   \phi_0(w), \,\,\forall\,\,l\in\Z.\end{equation} We compute  $\phi(d_m\cdot (w\otimes
   t^{l-m}))=d_m\cdot \phi(w\otimes t^{l-m})$:
$$d_m\cdot \phi(w\otimes
t^{l-m})=\lambda^{m-l}d_m\cdot (\phi_0(w)\otimes t^{l-m})$$
$$=\lambda^{m-l}(\lambda^{-m}(d_{-1}+mb_0')-d_{-1}+a+l-m+b_0m)\phi_0(w)\otimes
t^l$$
$$=(\lambda^{-l}(d_{-1}+mb_0')+\lambda^{m-l}(-d_{-1}+a+l+(b_0-1)m)\phi_0(w)\otimes
t^l,$$
$$\phi(d_m\cdot (w\otimes
   t^{l-m}))=\phi((\lambda^m(d_{-1}+mb')-d_{-1}+a+l +(b-1)m)w\otimes
t^l)$$
$$=(\lambda^m-1)\phi_l(d_{-1}w)\otimes
t^l+(m\lambda^mb'+a+l+(b-1)m)\lambda^{-l}\phi_0(w)\otimes t^l.$$
Comparing the coefficients of $m$, $m\lambda^m$ we
obtain that
\begin{equation}
b'-b_0+1=b_0'-b+1=0. \end{equation}
Thus (ii) holds in this case.

 \

 {\bf Case 2}. $\lambda=\lambda_0=-1$.

Note that
   $\phi(\Soc_{\mathfrak{b}}(W)\otimes \C[t,t^{-1}])\subset \Soc_{\mathfrak{b}}(W_0)\otimes
   \C[t,t^{-1}]$.
 If $W$ is not a highest weight module, again by comparing with
   the coefficient of $m^{2r+2}$ in (5.2) and using (5.3) and (5.4),
   we have $\phi_{l+n}(d_r^2
   w)=d_r^2\psi_{n}(w),$ for all $n,l\in \Z$, and $w\in
   \Soc_{\mathfrak{b}}(W)$. Thus $M\ne 0$ in this case.

Now suppose that $W$ is a highest weight module  the highest weight vector $w$. So $r=r'=0$. Let
where $b',b_0'$ be highest weight of $W,W_0$ respectively. Then $b'b'_0\ne0$ because of the simplicity of the modules. From
(5.3)-(5.6), we may deduce that $\phi_{2k}(w)=\phi_0(w)$,
$\phi_{2k+1}(w)=\phi_1(w)$.
We compute
$\phi(d_{2m+1}\cdot (w\otimes
  t^{i_0-2m+2l}))=d_{2m+1}\cdot(\phi_{i_0} (w)\otimes
  t^{i_0-2m+2l})$:
$$\aligned &d_{2m+1}\cdot(\phi_{i_0} (w)\otimes
  t^{i_0-2m+2l})\\=&(-2d_{-1}+(2m+1)(b_0-b_0')+a+i_0-2m+2l)\phi_{i_0}(w)\otimes
  t^{i_0+2l+1},\endaligned$$
  $$\aligned &\phi(d_{2m+1}\cdot (w\otimes
  t^{i_0-2m+2l}))\\=&\phi(-2d_{-1}+(2m+1)(b-b')+a+i_0-2m+2l)w\otimes
  t^{i_0+2l+1})\\=&((2m+1)(b-b')+a+i_0-2m+2l)\phi_{i_0+1}(w)\otimes
  t^{i_0+2l+1})\\ &-2\phi(d_{-1}(w)\otimes
  t^{i_0+2l+1}).\endaligned$$
Comparing the coefficient  of $m$,  we
obtain
$$(b_0-b_0'-1)\psi_{i_0}(w)=(b-b'-1)\psi_{i_0+1}(w), \forall i_0\in
\Z,$$
 Since $\mathcal{L}(W,\lambda, a,b)$ is simple, $b-b'-1\ne0$. By
taking $i_0=0$ and $1$ we obtain that  $b-b'-1=\pm(b_0-b'_0-1)$.
Thus we have either
$$\phi_{i}(w)=\phi_{0}(w), \forall i\in \Z{\text{ or }}$$ $$\phi_i(w)=(-1)^i \phi_0(w),\forall i\in
\Z.$$  For the first case we have $ M\ne 0$. Now we consider the
second case, i.e., $\phi_i(w)=(-1)^i \phi_0(w)$ for all $i\in\Z$. By
a same argument  as in Subcase 1.2, we can prove that (ii) holds in  this case.

\

(b). The sufficiency is obvious. We need only  to consider the case
$\mathcal{L}_0(a_1,b_1)\cong \mathcal{L}_0(a_2,b_2)$. Now suppose
$\psi:\mathcal{L}_0(a_1,b_1)\rightarrow \mathcal{L}_0(a_2,b_2)$ is
an isomorphism.  Then it is clear that $a_1=a_2$. Denote
$\psi(v\otimes t^n)=\psi_n(v)\otimes t^n$. Again we have
(\ref{4.11}). We may consider  $\lambda=\lambda_0=-1$, $r=r'=0$ as
in the argument in (a). From (5.3)-(5.6), we have
$\psi_{2k}(w)=\psi_0(w)$, where $w$, $\psi_0(w)$ are the highest
weight vectors.  For any $k\in\Z$ we compute  $\psi(d_{2k} \cdot
(w\otimes t^{0}))=d_{2k}\cdot \psi(w\otimes t^{0})$:
$$\psi(d_{2k} \cdot (w\otimes t^{0}))=(2k b_1+a_1+2k(b_1+1))\psi_0(w)\otimes
t^{2k},$$
$$d_{2k} \cdot \psi((w\otimes t^{0}))=(2k b_2+a_1+2k(b_2+1))\psi_0(w)\otimes
t^{2k},$$ yielding  that $b_1=b_2$.

\

(c). Since  $\mathcal{L}(W,\lambda, a,b)$ is simple, then
$\lambda\ne1$ and $b\ne1$. Let $\phi:  \mathcal{L}(W,\lambda,
a,b)\to \mathcal{N}(B,a_0)$ be a module isomorphism. We know that
$a=a_0$. Since $W$ is nontrivial, the action of $\C[d_{-1}]$ on $W$
is torsion-free. Let $r=\ord_\mathfrak{b}(\Soc_\mathfrak{b}(W))$ and
$r'=\ord_\mathfrak{b}(B)$. We also define $\phi_n(w\otimes
t^n)=\phi_n(w)\otimes t^n$ for any $w\in \Soc_\mathfrak{b}(W)$ and
any $n\in\Z$. Then $\phi_n:W\to B$ is a vector space isomorphism. We
can have a similar equation as in (\ref{4.11}), while one side comes
from $d_{l-m}d_m(v(n))$ in the proof of Lemma 3 in \cite{LLZ}. From
(5.3) and (5.4), we see that $b=1$ which is impossible. 
So $\mathcal{L}(W,\lambda, a,b)$ and $\mathcal{N}(B,a_0)$ cannot be
isomorphic.

To consider the isomorphisms between $\mathcal{L}(W,\lambda, a,b)$
and $\mathcal{L}_0(a_1,b_1)$ (or $\mathcal{L}_1(a_1,b_1)$), we let
$\phi:  \mathcal{L}(W,\lambda, a,b)\to \mathcal{L}(W',-1, a,b_1+1)$
be a nonzero one-to-one module homomorphism where $W'$ is the
highest weight $\mathfrak{W}$-module with highest weight $b_1$. Note
that $\mathcal{L}_0(a_1,b_1)$ and $\mathcal{L}_1(a_1,b_1)$ are the
only simple submodules of $\mathcal{L}(W',-1, a,b_1+1)$. We also
define $\phi_n(w\otimes t^n)=\phi_n(w)\otimes t^n$ for any $w\in
\Soc_\mathfrak{b}(W)$ and any $n\in\Z$. Then $\phi_n:W\to W'$ is a
one-to-one vector space   homomorphism.  As in the argument after
(5.1) we deduce that $\lambda=-1$. By a similar argument to Case 2
in the proof for (a), we can see that $W$ is also a highest weight
module. Continuing the argument as in Case 2 of the proof of (a) we see that $b=b'+1$ which contradicts the simplicity of $\mathcal{L}(W,\lambda, a,b)$. So there is no such $\phi$ exists. Thus
$\mathcal{L}(W,\lambda, a,b)$ and $\mathcal{L}_0(a_1,b_1)$ (or
$\mathcal{L}_1(a_1,b_1)$) cannot be isomorphic.

At last we consider isomorphisms between $\mathcal{N}(B,a_0)$ and
$\mathcal{L}_0(a_1,b_1)$ (or $\mathcal{L}_1(a_1,b_1)$). Let $\phi:
\mathcal{N}(B,a_0)\to \mathcal{L}(W',-1, a,b_1+1)$ be a nonzero
one-to-one module homomorphism  where $W'$ is the highest weight
$\mathfrak{W}$-module with highest weight $b_1$. We also define
$\phi_n(w\otimes t^n)=\phi_n(w)\otimes t^n$ for any $w\in B$ and any
$n\in\Z$. Then $\phi_n:B\to W'$ is a one-to-one vector space
homomorphism. We can have a similar equation as in (\ref{4.11}).
From (5.3) and (5.4), we deduce that $b_1=0$ which is impossible. So
$\mathcal{N}(B,a_0)$ and $\mathcal{L}_0(a_1,b_1)$ (or
$\mathcal{L}_1(a_1,b_1)$) cannot be isomorphic.
\end{proof}

\section{Virasoro modules $\mathcal{L}(W,\lambda, a,b)$ are new}

We need only to compare our simple Virasoro modules
$\mathcal{L}(W,\lambda, a,b)$ and $\mathcal{L}_0( a,b)$ with the simple Virasoro modules
obtained in \cite{CGZ}. Let us first recall the modules in \cite{CGZ}.

Let $U:=U(\mathfrak{V})$ be the universal enveloping algebra of the
Virasoro algebra $\mathfrak{V}$. For any $\dot c, h\in \C$, let
$I(\dot c,h)$ be the left ideal of $U$ generated by the set $$
\bigl\{d_{i}\bigm|i>0\bigr\}\bigcup\bigl\{d_0-h\cdot 1, c-\dot
c\cdot 1\bigr\}. $$ The Verma module with highest weight $(\dot c,
h)$ for $\mathfrak{V} $ is defined as the quotient  $\bar V(\dot
c,h):=U/I(\dot c,h)$. It is a highest weight module of $\mathfrak{V}
$ and has a basis consisting of all vectors of the form $$
d_{-i_1}d_{-i_2}\cdots d_{-i_k}v_{h};\quad k\in{\N}\cup\{0\},
i_{j}\in\N, i_{k}\geq\cdots\geq i_2\geq i_1>0.
$$ Then we have the {\it simple highest weigh module} $ V(\dot c,h)=\bar V(\dot c,h)/J$
where $J$ is the maximal proper submodule of $\bar V(\dot c,h)$.

\begin{theorem} Let $\lambda \in \C^*, a, b,\dot c, h, a_1, b_1\in
\C$,
and let $W\in \mathcal{O}_{\mathfrak{W}}$ be nontrivial simple. Then
simple modules $\mathcal{L}(W,\lambda, a,b), \mathcal{L}_0(a,b),
\mathcal{L}_1(a,b)$ are not isomorphic to any simple submodules of $
V(\dot c,h)\otimes A'_{a_1,b_1}$
\end{theorem}

\begin{proof} Let us introduce the operator in $U(\mathfrak{V})$:
$$X_{l,m}=d_{l-m-3}d_{m+3}-3d_{l-m-2}d_{m+2}+3d_{l-m-1}d_{m+1}-d_{l-m}d_{m}, \forall l,m\in\Z.$$

Let $v_1$ be the highest weight vector of $V(\dot c,h)$. From
\cite{CGZ}, there is a nonzero vector in any simple submodules of $
V(\dot c,h)\otimes A'_{a_1,b_1}$ of the form $v_1\otimes v_2$ for
some weight vector
  $v_2\in A'_{a_1,b_1}$.
From the proof of Theorem 7 in \cite{LLZ}, we know that
$$X_{l,m} (v_1\otimes v_2)= v_1\otimes \omega_{l,m}^{(3)}v_2=0, \forall\ m>0, l>m+3.$$
For any nonzero weight vector $w\otimes t^k\in
\mathcal{L}(W,\lambda, a,b),$ or $ \mathcal{L}_0(a,b),$ or $
\mathcal{L}_1(a,b)$, where $w\in W, k\in \Z$, one can compute that
$$X_{l,m} (w\otimes
t^k)\in (\lambda-1)^3(\lambda^{l-m-3}-\lambda^m)d_{-1}^2w\otimes
t^{l+k}+(d_{-1}U(\mathfrak{b})+U(\mathfrak{b}))w\otimes t^{l+k}.$$
Recall that $W\cong
\Ind_{\mathfrak{b}}^{\mathfrak{W}}(\Soc_{\mathfrak{b}}W)=\C[d_{-1}]\otimes
(\Soc_{\mathfrak{b}}W)$. So $X_{l,m} (w\otimes t^k)$ is nonzero for
$l>m+3$ with $(\lambda-1)^3(\lambda^{l-m-3}-\lambda^m)\ne 0$. Thus
the theorem follows.
\end{proof}

\noindent {\bf Acknowledgement.} The second author is partially
supported by NSF of China (Grant 11271109) and NSERC. Part of the
work in this paper was carried out when both authors were visiting
Institute of Mathematics, Chinese Academy of Sciences, Beijing,
China and Chern Institute of Mathematics, Nankai University,
Tianjin, China in the summer of 2012.

\vspace{1cm}

\noindent  R.L.: Department of Mathematics, Soochow university,
Suzhou 215006, Jiangsu, P. R. China.
 Email: rencail@amss.ac.cn

\vspace{0.2cm} \noindent K.Z.: Department of Mathematics, Wilfrid
Laurier University, Waterloo, ON, Canada N2L 3C5,  and College of
Mathematics and Information Science, Hebei Normal (Teachers)
University, Shijiazhuang, Hebei, 050016 P. R. China. Email:
kzhao@wlu.ca


\begin{thebibliography}{99999}
\bibitem[Bl]{Bl} R.~Block; The irreducible representations of the Lie algebra $\mathfrak{sl}(2)$ and of the Weyl
algebra. {\it Adv. in Math.} {\bf 139} (1981), no. 1, 69--110.
\bibitem[CGZ]{CGZ} H.~Chen, X.~Guo and K.~Zhao; Tensor product weight modules over the Virasoro
algebra. arXiv:1301.0526.
\bibitem[CM]{CM}
C. Conley, C. Martin; A family of irreducible representations of the
Witt Lie algebra with infinite-dimensional weight spaces. {\it
Compos. Math.}, 128(2),  153-175(2001).
\bibitem[FF]{FF} B.~Feigin, D.~Fuks; Verma modules over a Virasoro algebra. {\it Funktsional. Anal. i
Prilozhen.}
{\bf 17} (1983), no. 3, 91--92.
\bibitem[FJK]{FJK} E.~Felinska, Z.~Jaskolski, M.~Kosztolowicz; Whittaker pairs for the Virasoro algebra and the
Gaiotto-Bonelli-Maruyoshi-Tanzini states. {\it J. Math. Phys.} {\bf
53} (2012), 033504.
\bibitem[FMS]{FMS} P. Di Francesco, P. Mathieu, D. Senechal;
{\it Conformal Field Theory}, Springer, New York, 1997.
\bibitem[GO]{GO} P. Goddard, D. Olive; Kac-Moody and Virasoro algebras
in relation to quantum physics, {\it Intemat. J. Mod. Phys. A},
(1986), 303-414.
\bibitem[GLZ]{GLZ} X.~Guo, R.~L{\"u}, K.~Zhao; Fraction representations and highest-weight-like
representations of the Virasoro algebra. Preprint, 2010.
\bibitem[IK]{IK} K.~Iohara, Y.~Koga; {\it Representation theory of the Virasoro
algebra.}
Springer Monographs in Mathematics. Springer-Verlag London, Ltd.,
London, 2011.
\bibitem[IKU]{IKU} T. Inami, H. Kanno, T. Ueno; Higher-dimensional WZW model on K¨ahler manifold and
toroidal Lie algebra. {\it Mod. Phys. Lett} A 12(1997), 2757-2764.
\bibitem[IKUX]{IKUX}  T. Inami, H. Kanno, T. Ueno,  C.-S. Xiong; Two toroidal Lie algebra as current algebra of
four dimensional Kahler WZW model.
 {\it Phys. Lett} B, 399(1997), 97-104.
\bibitem[K]{K}  V. Kac; \emph{Infinite dimensional Lie algebras, 3rd edition},
Cambridge Univ. Press, 1990.
\bibitem[KR]{KR} V.~Kac, A.~Raina; {\it Bombay lectures on highest weight representations of infinite-dimensional
Lie algebras}. Advanced Series in Mathematical Physics, {\bf 2}.
World Scientific Publishing Co., Inc., Teaneck, NJ, 1987.
\bibitem[LL]{LL} J.
Lepowsky, H. Li; {\it Introduction to Vertex Operator Algebras and
Their Representations},  Birkhauser, 2004.
\bibitem[LLZ]{LLZ} G.~Liu, R.~L{\"u} and K.~Zhao; A class of simple weight Virasoro modules,
Preprint, arXiv:1211.0998.
\bibitem[LGZ]{LGZ} R.~L{\"u}, X.~Guo, K.~Zhao; Irreducible modules over the Virasoro algebra. {\it Doc. Math.} {\bf 16}
(2011), 709--721.
\bibitem[LZ]{LZ} R.~L{\"u},  K.~Zhao; Irreducible  Virasoro modules from irreducible Weyl modules, arXiv:1209.3746.
\bibitem[Mt]{Mt} O.~Mathieu; Classification of Harish-Chandra modules over the Virasoro Lie algebra. {\it Invent.
Math.} {\bf 107} (1992), no. 2, 225--234.
\bibitem[Mz]{Ma} V.~Mazorchuk; {\it Lectures on $\mathfrak{sl}_2(\mathbb{C})$-modules}. Imperial College Press,
London, 2010.
\bibitem[MZ1]{MZ1} V.~Mazorchuk, K.~Zhao;  Classification of simple weight Virasoro modules with a
finite-dimensional weight space. {\it J. Algebra} {\bf 307} (2007),
no. 1, 209--214.
\bibitem[MZ2]{MZ3} V. Mazorchuk, K. Zhao; Simple Virasoro modules which
are locally finite over a positive part. arXiv:1205.5937v1.
\bibitem[MW]{MW} V. Mazorchuk, E. Wiesner; Simple Virasoro
modules induced from codimension one subalgebras of the positive
part. arXiv:1209.1691
\bibitem[MoP]
{MoP} R. V. Moody,  A. Pianzola;  {\it Lie algebras with triangular
decompositions}, Canad. Math. Soc., Ser. Mono. Adv. Texts, A
Wiley-Interscience Publication, John Wiley \& Sons Inc., New York,
1995.
\bibitem[OW1]{OW} M.~Ondrus, E.~Wiesner; Whittaker modules for the Virasoro algebra. {\it J. Algebra
Appl.}
{\bf 8} (2009), no. 3, 363--377.
\bibitem[OW2]{OW2} M.~Ondrus, E.~Wiesner; Whittaker categories for the Virasoro algebra.
Preprint arXiv:1108.2698.
\bibitem[TZ]{TZ} H.~ Tan, K.~Zhao; Irreducible modules from tensor produces, arXiv:1301.2131.
\bibitem[Ya]{Ya} S.~Yanagida; Whittaker vectors of the Virasoro algebra in terms of Jack symmetric polynomial.
{\it J. Algebra} {\bf 333} (2011), 273--294.
\bibitem[Zh]{Zh} H.~Zhang; A class of representations over the Virasoro algebra. {\it J. Algebra} {\bf 190}
(1997), no. 1, 1--10.
\bibitem[Zk]{Zk} K. Zhao; Representations of the
Virasoro algebra I, {\it J. Algebra,} {\bf 176}(1995), no. 3,
882-907.
\end{thebibliography}
\end{document}